\documentclass[10pt,a4paper]{article}
\usepackage[utf8]{inputenc}
\usepackage[T1]{fontenc}
\usepackage[english]{babel}
\usepackage{amsmath}
\usepackage{amsfonts}
\usepackage{amssymb}
\usepackage{amsthm}
\usepackage{makeidx}
\usepackage{graphicx}
\usepackage{xcolor}
\usepackage{tikz}

\theoremstyle{theorem}
\newtheorem{theorem}{Theorem}

\newtheorem{corollary}{Corollary}
\newtheorem{pro}{Proposition}

\newtheorem{note}{Note}

\theoremstyle{definition}
\newtheorem{definition}{Definition}
\newtheorem*{remark}{Remark}
\newtheorem{example}{Example}

\begin{document}
	
	\title{\textbf{Statistical Compactness}}
	\author{Manoranjan Singha and Ujjal Kumar Hom}
	
	\date{}
	\maketitle
	{\let\thefootnote\relax\footnotetext{{MSC: Primary 54A20, Secondary 40A35.\\Department of Mathematics, University of North Bengal, Raja Rammohunpur, Darjeeling-734013, West Bengal, India.\\ Email address: manoranjan.math@nbu.ac.in, rs\_ujjal@nbu.ac.in}}}
	\begin{abstract}
		Organising the relevant literature and by letting statistical convergence play the main role in the theory of compactness, a variant of compactness called statistical compactness has been achieved. As in case of sequential compactness, one point statistical compactification is studied to some extent too.
	\end{abstract}

	\noindent
\section{Introduction}
The idea of statistical convergence of real numbers was introduced by H. Fast in \cite{1} and H. Steinhaus in \cite{2}. Later this idea is generalized and exhibited in many papers (e.g. \cite{4},\cite{5},\cite{6},\cite{7},\cite{8},\cite{9},\cite{10},\cite{11},\cite{12},\cite{14}).

The concept of statistical convergence is an extension of the usual convergence of sequence and is based on the notion of asymptotic density \cite{14} of subset of natural numbers $\mathbb{N}$. If $A\subset \mathbb{N}$, denote the cardinality of $A$ by $|A|$ and $d_n(A)=\frac{|\{m\in \mathbb{N}:m\in A\cap\{1,2,...,n\} \}|}{n}$. The numbers
\begin{center}
	$\underline{d}(A) = \displaystyle{\liminf_{n\rightarrow\infty}}$ $d_n(A)$ and $\overline{d}(A) = \displaystyle{\limsup_{n\rightarrow\infty}}$ $d_n(A)$
\end{center}
are called the lower and upper asymptotic density of A, respectively. If $\underline{d}(A) = \overline{d}(A)$, then $d(A)=\overline{d}(A)$ is called asymptotic density or natural density of $A$. As defined by Fridy in \cite{9}, a subsequence $(x_n)_{n\in K}$ of $(x_n)_{n\in \mathbb{N}}$ is called thin subsequence if $d(K)=0$ otherwise $(x_n)_{n\in K}$ is called nonthin subsequence of $(x_n)_{n\in \mathbb{N}}$. In \cite{3}, Brown introduced one point sequential compactification. In this paper statistical compactness, a variant of compactness where statistical convergence of nonthin subsequences plays the prime role, is defined and the notion of one point statistical compactification is developed using statistical compact sets.

	\section{Main results}
	Let's begin with a difference: unlike usual convergence, even nonthin subsequence of a statiatically convergent sequence may fail to be statistical convergent. For, let's define a sequence $(\mathbf{a}_n)_{n\in \mathbb{N}}$ as follows:\newline
	Suppose $A=\displaystyle{\bigcup_{k=2}^{\infty}A_k}$ where $A_k=\{k^{k^{2}}+1,k^{k^{2}}+2,...,k^{k^{2}+1}\}$. Since $\displaystyle{\lim_{k\to \infty}}d_{k^{k^{2}}+1}(A)$\\$=0$ and $\displaystyle{\lim_{k\to \infty}d_{k^{k^{2}+1}}(A)=1}$, $\underline{d}(A)=0$ and $\overline{d}(A)=1$ respectively. \newline
	Define a strictly increasing function $f:\mathbb{N}\to \mathbb{N}$ by $f(i) = i$, $1\leqslant i\leqslant 16$ and $k\geqslant 2$, $f(k^{k^{2}}+n) = {(k+1)}^{k^{2}+3}+n$, $1\leqslant n\leqslant {(k+1)}^{{(k+1)}^{2}}-k^{k^{2}}$. For ${(n+1)}^{n^{2}+3}+1\leqslant k\leqslant {(n+2)}^{{(n+1)}^{2}+3}$,\begin{center}
		$d_k(A)\leqslant \frac{\displaystyle{\sum_{r=2}^{n}}r^{r^2}(r-1)}{{(n+1)}^{n^{2}+3}+1}$
	\end{center}
	which follows that $d(f(A))=0$. $\overline{d}(f(\mathbb{N}))=1$ as $\displaystyle{\lim_{n\to \infty}d_{r_n}(f(\mathbb{N}))=1}$ where $r_n={(n+1)}^{n^{2}+3}+{(n+1)}^{{(n+1)}^{2}}-n^{n^{2}}$. Define \begin{center}
		$\mathbf{a}_n=\begin{cases}0,& \text{if }n\in f(A)
			\\ 1, & \text{otherwise } \end{cases} $
	\end{center}
	Then $(\mathbf{a}_n)_{n\in \mathbb{N}}$ is statistically convergent to 1 but the nonthin subsequence $(\mathbf{a}_{f(k)})_{k\in \mathbb{N}}$ is not statistically convergent. Thus nonthin subsequence of a statistically convergent sequence may not be statistically convergent. This barrier can be removed in the following way:
	\begin{itemize}
		\item A sequence is a mapping whose domain is cofinal subset of $\mathbb{N}$. Suppose $(a_n)_{n\in M}$ is sequence in a topological space $X$ and $N$ is a cofinal subset of $M$. Call $(a_n)_{n\in N}$ is a subsequence of $(a_n)_{n\in M}$.
		\item Let's call a nonthin sequence $(a_n)_{n\in M}$ in a topological space $X$ is statistically convergent to $a\in X$ if for any open subset $U$ of $X$ containing $a$, $d(\{n\in M:a_n\notin U\})=0$.
	\end{itemize}

The following \textbf{Note \ref{n1}} also shows the urge of the above two definitions.
\begin{note}\label{n1}
Define a strictly increasing function $g:\mathbb{N}\to \mathbb{N}$ by $g(i) = i$, 1$\leqslant i\leqslant 16$ and for $k\geqslant 2$, $g(k^{k^{2}+1}+n) = {(k+1)}^{(k+1)^{2}+1}+n$, $1\leqslant n\leqslant {(k+1)}^{{(k+1)}^{2}}-k^{k^{2}+1}$ and $g(k^{k^{2}}+n) = {k}^{k^{2}+1}+k^{k^2}-{(k-1)}^{(k-1)^{2}+1}+n$, $1\leqslant n\leqslant k^{k^{2}}(k-1)$. Then $d(g(\mathbb{N}\backslash A))=0$ and $\displaystyle{\lim_{n\to \infty}}d_{s_n}(g(\mathbb{N}))=\frac{1}{2}$ where $s_n={2(n+1)}^{(n+1)^{2}+1}-n^{n^{2}+1}$. Define \begin{center}
	$\mathbf{b}_n=\begin{cases}0,& \text{if }n\in g(A)
		\\ 1, & \text{otherwise } \end{cases} $
\end{center} and \begin{center}
$\mathbf{x}_n=\begin{cases}0,& \text{if }n\in A
	\\ 1, & \text{otherwise. } \end{cases} $
\end{center}Now a fact is $\mathbf{x}_k=\mathbf{a}_{f(k)}=\mathbf{b}_{g(k)}$ for all $k\in \mathbb{N}$ but
\begin{itemize}
	\item[$\ast$] $(\mathbf{x}_n)_{n\in \mathbb{N}}$ is not statistically convergent.
	\item[$\ast$] $(\mathbf{a}_n)_{n\in f(\mathbb{N})}$ is statistically convergent to 1.
	\item[$\ast$] $(\mathbf{b}_n)_{n\in g(\mathbb{N})}$ is statistically convergent to 0.
\end{itemize}
\end{note}

	\begin{theorem}\label{th1}
		Let $X$ be a first countable space and $(x_n)_{n\in M}$ be a nonthin sequence in $X$. Then $(x_n)_{n\in M}$ is statistically convergent to $x\in X$ if and only if there exists a subset $N$ of $M$ such that $\overline{d}(M)=\overline{d}(N)$ and $(x_n)_{n\in N}$ converges to $x$.
	\end{theorem}
	\begin{proof}
		Suppose $(x_n)_{n\in M}$ statistically converges to $x\in X$. Let $(U_n)_{n\in \mathbb{N}}$ be a sequence of open sets in $X$ such that $U_{n+1}\subset U_n$ and $x\in U_n$ for all $n\in \mathbb{N}$. Put $K_n =\{m\in M:x_m\in U_n\}$, $n\in \mathbb{N}$. Then $\overline{d}(M)=\overline{d}(K_n)$.\\Let us choose an arbitrary number $v_1\in K_1$ such that $d_{v_1}(K_1)>0$. Suppose $K_2=\{n_1<n_2<n_3<...\}$. Since $\overline{d}(K_2) = \displaystyle{\limsup_{r\rightarrow\infty}}$ $d_{n_r}(K_2)$,
		\begin{center}
			$|d_{n_r}(K_2)-\overline{d}(K_2)|<\frac{1}{2}$ for frequently many $r$
		\end{center}i.e.,
		\begin{center}
			$d_{n_r}(K_2)>\overline{d}(M)-\frac{1}{2}$ for frequently many $r$.
		\end{center}
		So there exists a $v_2\in K_2$ such that $v_2>v_1$ and $d_{v_2}(K_2)>\overline{d}(K_2)-\frac{1}{2}$. Thus one can construct by induction such a sequence $(v_n)_{n\in \mathbb{N}}$ of natural numbers such that $v_n\in K_n$ with $v_{n+1}>v_n$ and $d_{v_n}(K_n)>\overline{d}(M)-\frac{1}{n}$.\\Define $N=\displaystyle{\bigcup_{i=1}^{\infty}}\{v_{i-1},...,v_i-1\}\cap K_{i-1}$ where $v_0=1$ and $K_0=M$. $d_{v_n}(N)\geqslant d_{v_n}(K_n)>\overline{d}(M)-\frac{1}{n}$ for all n and this implies that $\overline{d}(M)\leqslant \overline{d}(N)$ i.e., $\overline{d}(N)=\overline{d}(M)$. Since $x_m\in U_n$ for all $m\in \displaystyle{\bigcup_{i=n+1}^{\infty}}\{v_{i-1},...,v_i-1\}\cap K_{i-1}$, $(x_n)_{n\in N}$ converges to $x$.
		
		Converse follows from the fact that $d(M\backslash N)=0$ if $\overline{d}(M)=\overline{d}(N)$.
	\end{proof}
	\begin{example}
		Let $J$= collection of all nonthin sunsets of $\mathbb{N}$. Then $J$ is uncountable. For $j\in J$, let $\mathcal{A}_j\subset j$ such that $\mathcal{A}_j$ is infinite and $d(\mathcal{A}_j)=0$.\\
		Consider the product space $X=\{0,1\}^J$ where $\{0,1\}$ is discrete space. Define a sequence $(x_n)_{n\in \mathbb{N}}$ in $X$ by 
		\begin{center}
			$\pi_j(x_n)=\begin{cases}1,& \text{if }n\in \mathcal{A}_j\\ 0, & \text{otherwise } \end{cases} $
		\end{center} 
		As $\{n\in \mathbb{N}: x_n\notin \pi_{j}^{-1}(\{0\})\}=\mathcal{A}_j$, $(x_n)_{n\in \mathbb{N}}$ statistically converges to $0\in X$. But for no nonthin subsequence $(x_n)_{n\in M}$ of $(x_n)_{n\in \mathbb{N}}$, $(\pi_M(x_n))_{n\in M}$ converges to 0. So no nonthin subsequence of $(x_n)_{n\in \mathbb{N}}$ converges to $0\in X$. This example shows that first countable space is necessary for Theorem \ref{th1}.
	\end{example}
	\begin{definition}
		Let $(X,\tau)$ be a topological space and let $F\subset X$. Define statistical closure of $F$ as the set $\{x\in X:$ there exists a nonthin sequence $(x_n)_{n\in M}$ in $F$ which is statistically convergent to $x\}$ and denote the set by $\overline{F}^{ST}$. Let's call $F$ is statistically closed if $\overline{F}^{ST}=F$.
	\end{definition}
	\begin{note}
		$F$ is statistically closed if $F$ is closed for $F\subset \overline{F}^{ST}\subset \overline{F}$. But there is no difference between closed and statistically closed subsets of $X$ if $X$ is first countable follows from Theorem \ref{th1}. 
	\end{note}
	\begin{note}
		$\tau_{ST}=\{F\subset X:X\backslash F$ is statistically closed$\}$ forms a topology on $X$ with $\tau\subset \tau_{ST}$.
	\end{note}
	\begin{definition}
		Let $(X,\tau)$ be a topological space. $X$ is called statistical sequential space if $\tau = \tau_{ST}$.
	\end{definition}
	\begin{definition}
		Let $X$ and $Y$ be two topological spaces and let $f:X\to Y$ be a function. $f$ is called statistically continuous function if for any nonthin sequence $(x_n)_{n\in K}$ in $X$ such that $(x_n)_{n\in K}$ statistically converges to $x$, $(f(x_n))_{n\in K}$ statistically converges to $f(x)$.
	\end{definition}
	\begin{note}
		Any continuous function is also statistically continuous.
	\end{note}
	\begin{theorem}
		Let $X$ and $Y$ be two topological spaces and let $f:X\to Y$ be a function. $f$ is statistically continuous if and only if $f^{-1}(B)$ is statistically closed for any statistically closed subset $B$ of $Y$.
	\end{theorem}
	\begin{proof}
		Let $f$ be statistically continuous and let $B$ be a statistically closed subset of $Y$. Suppose $x\in \overline{f^{-1}(B)}^{ST}$. There exists a nonthin sequence $(x_n)_{n\in K}$ in $X$ such that $(x_n)_{n\in K}$ statistically converges to $x\in X$. So $(f(x_n))_{n\in K}$ statistically converges to $f(x)$. Therefore $x\in f^{-1}(B)$ because $B$ is statistically closed.\\
		Conversely let $f^{-1}(B)$ is statistically closed for any statistically closed subset $B$ of $Y$. Suppose $(x_n)_{n\in K}$ is a nonthin sequence in $X$ such that $(x_n)_{n\in K}$ statistically converges to $x\in X$ and $U$ is an open set in $Y$ with $f(x)\in U$. If possible let $\overline{d}(K')>0$ where $K'=\{n\in K:f(x_n)\notin U\}$. Since $Y\backslash U$ is closed in $Y$, $Y\backslash U$ is statistically closed in $Y$ and so $f^{-1}(Y\backslash U)$ is statistically closed in $X$. Since $(x_n)_{n\in K'}$ statistically converges to $x$ , $x\in \overline{(f^{-1}(Y\backslash U))}^{ST}=f^{-1}(Y\backslash U)$. Hence $f(x)\notin U$, which is a contradiction.
	\end{proof}
	\begin{theorem}
		Let $X=\displaystyle{\prod_{\lambda \in \Lambda}}X_{\lambda}$ be the product space of topological spaces $X_{\lambda}$ and let $(x_n)_{n\in M}$ be nonthin sequence in $X$ and $x\in X$. Then $(x_n)_{n\in M}$ statistically converges to $x$ if and only if $(\pi_{\lambda}(x_n))_{n\in M}$ statistically converges to $\pi_{\lambda}(x)$ for all $\lambda \in \Lambda$.
	\end{theorem}
	\begin{proof}
		Since projection map $\pi_{\lambda}:X\to X_{\lambda}$ is continuous, $(\pi_{\lambda}(x_n))_{n\in M}$ statistically converges to $\pi_{\lambda}(x)$ if $(x_n)_{n\in M}$ statistically converges to $x$. 
	\end{proof}
	\begin{corollary}
		Let $X$ be a topological space. Nonthin sequence in $X$ can statistically convergent to atmost one point of $X$ if and only if $\Delta X=\{(x,x): x\in X\}$ is statistically closed in $X\times X$.
	\end{corollary}
\begin{definition}
	A function $f:X\to Y$ is statistically closed if $f(B)$ is statistically closed for any statistically closed subset $B$ of $X$.
\end{definition}
	\begin{definition}
		A statistically continuous function $f:X\to Y$ is statistically proper if $f\times 1_Z:X\times Z\to Y\times Z$ is statistically closed for all spaces $Z$.
	\end{definition}
	\begin{definition}
		A bijective mapping $f:X\to Y$ is called statistical homeomorphism if $f$ and $f^{-1}$ are both statistical continuous.
	\end{definition}
	\begin{pro}
		The following are equivalent for an one-one statistical continuous function $f:X\to Y$:
		\begin{itemize}
			\item f is statistically proper,
			\item f is statistically closed,
			\item f is a statistical homeomorphism.
		\end{itemize}
	\end{pro}
	\begin{definition}
		A topological space X is called statistically compact if every non-thin sequence in $X$ has nonthin statistically convergent subsequence.
	\end{definition}
	\begin{theorem}
		Statistcally compact first countable space is sequentially compact.
	\end{theorem}
	\begin{proof}
		Let $X$ be a first countable space which is statistically compact and let $(x_n)_{n\in M}$ be a nonthin sequence in $X$.
		Then $(x_n)_{n\in \mathbb{N}}$ has a nonthin statistically convergent subsequence $(x_n)_{n\in N}$ that statistically converges to $a\in X$ . From Theorem \ref{th1} it follows that $(x_n)_{n\in N}$ has a convergent subsequence $(x_n)_{n\in K}$ such that $(x_n)_{n\in K}$converges to $a$. Hence $X$ is sequentially compact.
	\end{proof}
	\begin{corollary}
		Statistically compact metric space is compact.
	\end{corollary}
	\begin{example}
		Consiser the space $[1,\omega_1]$ with order topology where $\omega_1$ is the first uncountable ordinal. Let $(x_n)_{n\in A}$ be a nonthin sequence in $S_\Omega=[1,\omega_1)$. If range of $(x_n)_{n\in A}$ is finite then $(x_n)_{n\in A}$ has nonthin statistically convergent subsequence. Suppose range of $(x_n)_{n\in A}$ is not finite. Let $b\in S_\Omega$ be the least upper bound of $\{x_n:n\in A\}$. Let $[1,b)=\{y_m:m\in \mathbb{N}\}$ where $y_1=1$ and $y_{m+1}$ is the least upper bound of $[1,b)-\{y_1,...y_m\}$.\\ Define $S_m = \{n\in A: x_n = y_m\}$, $ m\in \mathbb{N} $ and $ S=\{n\in A:x_n = b\}$. If $\overline{\rm d}(S_m) \ne 0$ for some $m\in \mathbb{N}$ or $\overline{\rm d}(S) \ne 0$ then $(x_n)_{n\in S_m}$ for some $m\in \mathbb{N}$  or $(x_n)_{n\in S}$ become a nonthin statistically convergent subsequence of $(x_n)_{n\in A}$.\\Suppose $d(S_m) = 0$ for all $m\in \mathbb{N}$ and $d(S) = 0$ . Let  $\alpha \in [1,b)$. Then there exist a $p\in \mathbb{N}$ such that $\alpha=y_p$ and hence $y_n\in (\alpha,b
		]$ for all $n>p$. Therefore $\{n\in A:x_n\notin (\alpha,b]\}\subset$${\overset{p}{ \underset{i=1} \bigcup}  \{n\in A: x_n = y_i\}} ={\overset{p}{ \underset{i=1} \bigcup} S_i} $. Since $d(S_i)=0$ for all $i=1(1)p$, $d(\{n\in A:y_n\notin (\alpha,b]\})$ =0. Hence $(x_n)_{n\in A}$ is statistically converges to $b\in S_\Omega$. Therefore $S_\Omega$ is statistically compact space but not compact.
	\end{example}
	\begin{example}
		Consider $X=\{\frac{1}{n}:n\in \mathbb {N}\}\cup\{0\}$  as a subspace of $\mathbb{R}$ with usual topology. Then $X$ is statistically compact. But $X^{\prime} = \{\frac{1}{n}:n\in \mathbb{N}\}$ is a open subspace of X which is not compact as a subspace of X. Therefore $X^{\prime}$ is not statistically compact. So subspace of a statistically compact space may not be statistically compact.
	\end{example}
	\begin{theorem}
		Statistically closed subspace of statstically compact space is statistically compact.
	\end{theorem}
	\begin{proof}
		Let $X$ be a statistically compact space and let $Y$ be statistically closed subspace of $X$. Suppose $(x_n)_{n\in A}$  be a non-thin sequence in $Y$. There exists a nonthin subsequence $(x_n)_{n\in A^{\prime}}$ of $(x_n)_{n\in A}$ such that $(x_n)_{n\in A^{\prime}}$ statistically converges to $x\in X$. Therefore $x\in \overline{Y}^{ST} = Y $. Thus $(x_n)_{n\in A^{\prime}}$ is statistically converges to $x\in Y$ and so Y is statistically compact.
	\end{proof}
	\begin{corollary}
		Closed subspace of statstically compact space is statistically compact.
	\end{corollary}
	\begin{theorem}\label{th4}
		Suppose $X$ is a topological space such that $\Delta X$ is statistically closed. Then statistically compact subspace of $X$ is statistically closed.
	\end{theorem}
	\begin{proof}
		Let $Y$ be a statistically compact subspace of $X$. Suppose $x\in \overline{Y}^{ST}$. There exists a nonthin sequence $(x_n)_{n\in K}$ in $Y$ such that $(x_n)_{n\in K}$ is statistically convergent to $x$. For $Y$ is statistically compact, there is a nonthin subsequence $(x_n)_{n\in K'}$ of $(x_n)_{n\in K}$ so that $(x_n)_{n\in K'}$ is statistically convergent to $y$ for some $y\in Y$. Since $\Delta X$ is statistically closed, $x=y$.
	\end{proof}
	\begin{theorem}
		Statistical continuous image of a statistically compact is statistically compact.
	\end{theorem}
	\begin{proof}
		Let $X$ and $Y$ be two topological space where $X$ is statistically compact and let $f:X\to Y$ be statistical continuous onto map. \\Suppose $(y_n)_{n\in A}$ is a nonthin sequence in $Y$. There is $x_n\in X$ such that $f(x_n)=y_n$ for all $n\in A$. Also there exists a nonthin subsequence $(x_n)_{n\in A^{\prime}}$ of $(x_n)_{n\in A}$ such that $(x_n)_{n\in A^{\prime}}$ statistically converges to $x\in X$. Since $f$ is statistical continuous, $(y_n)_{n\in A^{\prime}}$ is statistically converges to $f(x)$. Hence $Y$ is statistically compact.
	\end{proof}
	\begin{corollary}\label{th2}
		Continuous image of a statistically compact is statistically compact.
	\end{corollary}
	\begin{example}\label{ex1}
		Define a sequence $(x_n)_{n\in \mathbb{N}}$  in $\mathbb{R}$ by $x_n=\frac{r}{m}$ if $n=\frac{m(m+1)}{2}+r ; r\in \{0,1,2,...,m\}$ and $m\in \mathbb{N}$. Let $A=\{n_1<n_2<...<n_k<...\}$ be a nonthin subset of $\mathbb{N}$. If possible let $(x_n)_{n\in A}$ is statistically convergent to $l$ for some $l\in [0,1]$.\\
		Let $l\in (0,1)$. Suppose $0<\epsilon<min\{l,1-l\}$. Since $\{n_k\leq n:x_{n_k}\in (l-\epsilon,l+\epsilon)\}\subset\{m\leqslant n:x_m\in [l-\epsilon,l+\epsilon]\}$ and $(x_n)_{n\in \mathbb{N}}$ is uniformly distributed sequence in $[0,1]$, $\overline{d}(\{n_k:x_{n_k}\in (l-\epsilon,l+\epsilon)\})\leqslant2\epsilon$. Therefore $\overline{d}(A)\leqslant2\epsilon$ for any $\epsilon$ satisfying $0<\epsilon<min\{l,1-l\}$ and this implies $\overline {d}(A)=0$, which is a contradiction. Similarly one can see a contradiction if $l\in \{0,1\}$. Thus $[0,1]$ is not statistically compact.
	\end{example}
	\begin{note}
		From Example \ref{ex1} and Corollary \ref{th2} it follows that any closed interval $[a,b]$, $a,b\in \mathbb{R}$ with $a<b$ is not statistically compact as a subspace of $\mathbb{R}$ and consequently no interval in $\mathbb{R}$ is statistically compact.
	\end{note}
	\begin{note}
		It follows from Example \ref{ex1} that sequentially compact space may not be statistically compact in first countable space and also compact space may not be statistically compact in metric space.
	\end{note}
	\begin{theorem}
		Finite product of statistically compact spaces is statistically compact.
	\end{theorem}
	\begin{proof}
		Let $X_i$ be a statistically compact space, $i=1,...n$ and let $X=\displaystyle\prod_{i=1}^{n}X_i$. Suppose $(x_m)_{m\in K}$ is a nonthin sequence in $X$. Since $X_1$ is statistically compact, there exists a nonthin subsequence $(\pi_1(x_m))_{m\in K_1}$ of $(\pi_1(x_m))_{m\in K}$ such that $(\pi_1(x_m))_{m\in K_1}$ is statistically converges to $x_1\in X_1$. As $X_2$ is statistically compact, there exists a nonthin subsequence $(\pi_2(x_m))_{m\in K_2}$ of $(\pi_2(x_m))_{m\in K_1}$ such that $(\pi_2(x_m))_{m\in K_2}$ is statistically converges to $x_2\in X_2$. Proceeding in this way we get $n$ nonthin sequences $(\pi_i(x_m))_{m\in K_i}$, $i=1,2,...,n$ such that $K_1\supset K_2\supset ...\supset K_n$ and $(\pi_i(x_m))_{m\in K_i}$ is statistically convergent to $x_i\in X_i$, $i=1,2,...,n$.\\Let $x=(x_1,x_2,...,x_n)\in X$ and let $U_i$ be open in $X_i$ such that $x_i\in U_i$. Since $\{m\in K_n:x_m\notin \displaystyle\bigcap_{i=1}^{n}\pi_{i}^{-1}(U_i)\}\subset \displaystyle\bigcup_{i=1}^{n}\{m\in K_i:\pi_i(x_m)\notin U_i\}$ and $d(\displaystyle\bigcup_{i=1}^{n}\{m\in K_i:\pi_i(x_m)\notin U_i\})=0$, $d(\{m\in K_n:x_m\notin \displaystyle\bigcap_{i=1}^{n}\pi_{i}^{-1}(U_i)\})=0$. Thus $(x_m)_{m\in K_n}$ is statistically converges to $x\in X$.
	\end{proof}
	\begin{example}\label{exx}
		Consider $X_n=\{1,...,n\}$ with discrete topology, $n\in \mathbb{N}$. Let $X=\displaystyle\prod_{n\in \mathbb{N}}^{}X_n$ be the product space. Define a sequence $(x_n)_{n\in \mathbb{N}}$ in $X$ such that for any $k\in \mathbb{N}$,
		\[ \pi_k(x_n)=\begin{cases}
			1,& \text{if }n=km,m\in \mathbb{N}\\2, & \text{if }n=km+1,m\in \mathbb{N}\cup \{0\}\\. &\\. &\\. &\\k, & \text{if }n=km+(k-1),m\in \mathbb{N}\cup \{0\} \end{cases} \]
		Suppose $(x_n)_{n\in A}$ is a nonthin subsequence of $(x_n)_{n\in \mathbb{N}}$. If possible let $(x_n)_{n\in A}$ be statistically convergent to $a=(a_m)_{m\in \mathbb{N}}\in X$. Since $d(\{m\in A:x_m\notin \pi_{i}^{-1}(\{a_i\})\})=0$, $\overline{d}(A)\leqslant\overline{d}(\{m\in A:x_m\in \pi_{i}^{-1}(\{a_i\})\})\leqslant\overline{d}(\{m\in \mathbb{N}:x_m\in \pi_{i}^{-1}(\{a_i\})\})=\frac{1}{i}$ for all $i\in \mathbb{N}$. Hence $\overline{d}(A)=0$, which is a contradiction. Therefore $X$ is not statistically compact.
	\end{example}
	\begin{note}\label{n}
		It follows from Example \ref{exx} that Cantor set is not statistically compact and so statistical compact subset of a seperable completely metrizable space is countable (see Corollary 3.6 in \cite{21}). 
	\end{note}
	\begin{theorem}
		Let $f:X\to Y$ be statistically continuous and let $\Delta Y$ be statistically closed. Then (a)$\implies$ (b)$\implies$ (c) where \\
		(a) If $(x_n)_{n\in M}$ is a nonthin sequence in $X$ with no nonthin subsequence statistically convergent in $X$ then so is $(f(x_n))_{n\in M}$ in $Y$.\\
		(b) If $B$ is statistically compact in $Y$ then so is $f^{-1}(B)$ in $X$.\\
		(c) f is statistically proper.
	\end{theorem}
	\begin{proof}
		(a)$\implies$ (b) Let $B$ be statistically compact subset of $Y$ and let $(x_n)_{n\in M}$ be a nonthin sequence in $f^{-1}(B)$. There exists a nonthin subsequence $(x_n)_{n\in N}$ of $(x_n)_{n\in M}$ such that $(f(x_n))_{n\in N}$ is statistically convergent to $y$ for some $y\in Y$. So $(x_n)_{n\in N}$ has a statistically convergent subsequence , say $(x_n)_{n\in P}$ which statistically converges to $x\in X$ and consequently $(f(x_n))_{n\in P}$ statistically converges to $f(x)$. Since $\Delta y$ is statistically closed, $f(x)=y$.\\
		(b)$\implies$ (c) Let $A$ be statistically closed in $X\times Z$. Suppose $(x_n,z_n)_{n\in M}$ is a nonthin sequence in $A$ such that $(f(x_n),z_n)_{n\in M}$ statistically converges to $(y,z)\in Y\times Z$. Since $\Delta y$ is statistically closed, $B=\{f(x_n): n\in M\}\cup \{y\}$ is statistically compact and hence so is $f^{-1}(B)$. There exists a nonthin subsequence $(x_n)_{n\in N}$ of $(x_n)_{n\in M}$ such that $(x_n)_{n\in N}$ is statistically convergent to $x\in f^{-1}(B)$. Then $(x,z)\in A$ because $(x_n,z_n)_{n\in N}$ statistically converges to $(x,z)$ and $A$ is statistically closed. So $(f(x),z)\in f(A)$ and $f(x)=y$ as $\Delta Y$ is statistically closed.
		\end{proof} 
	\begin{theorem}
		Let $(X,\tau$) be a topological space and $A\subset X$. Then $A$ is statistically compact in $(X,\tau)$ if and only if $A$ is statistically compact in $(X,\tau_{ST})$.  
	\end{theorem}
\begin{proof}
	Let $A$ be statistically compact in $(X,\tau)$ and let $(x_n)_{n\in M}$ be a nonthin sequence in $A$. Then there exists a nonthin subsequence $(x_n)_{n\in N}$ of $(x_n)_{n\in M}$ such that $(x_n)_{n\in N}$ statistically converges to $x\in A$ in $(X,\tau)$. Suppose $C$ be a statistically closed subset of $(X,\tau)$ such that $x\notin C$ and $P=\{n\in N: x_n\in C\}$. If $\overline{d}(P)>0$ then $(x_n)_{n\in P}$ will statistical convergent to $x$ in $(X,\tau)$ which will imply that $x\in C$. So $d(P)=0$. Thus  $A$ is statistically compact in $(X,\tau_{ST})$.\\ Converse follows directly as $\tau \subset \tau_{ST}$.
\end{proof}
	\begin{theorem}\label{th3}
		Let $X$ be a topological space and let $C\subset X$. $C$ is statistically compact implies that for any nonthin sequence $(x_n)_{n\in M}$ in $X$ which has no nonthin statistically convergent subsequence in $X$, there exists $N\subset M$ such that $d(N)=0$ and $x_n\in X\backslash C$ for all $n\in M\backslash N$. Converse holds if $C$ is statistically closed.
	\end{theorem}
	\begin{proof}
		Let $C$ be statistically compact. Suppose $(x_n)_{n\in M}$ is nonthin sequence in $X$ which has no nonthin statistically convergent subsequence in $X$. Then $d(N)=0$ and $x_n\in X\backslash C$ for all $n\in M\backslash N$ where $N=\{n\in M: x_n\in C\}$.\\ Suppose $C$ is statistically closed and converse holds. Let $(x_n)_{n\in M}$ be nonthin sequence in $C$. Since $x_n\in C$ for all $n\in M$, $(x_n)_{n\in M}$ has a nonthin statistically convergent subsequence, say $(x_n)_{n\in N}$ which statistically converges to $x$ for some $x\in X$. This implies $x\in \overline{C}^{ST}=C$. Therefore $C$ is statistically compact.
	\end{proof}
	\begin{theorem}
		Suppose $(X,\tau)$ is a topological space which is not statistically compact and $\infty^{X}$ is a point not in $X$. Define $X^s= X\cup \{\infty^{X}\}$ and $\tau_s^X= \tau \cup \{X^s\backslash C: C$ is closed and statistically compact subset of $X\}$. Then $(X^s,\tau_s^X)$ is statistically compact space.
	\end{theorem}
	\begin{proof}
		Since closed subspace of a statistically compact space is statistically compact and collection of all statistically compact subspace of $X$ is closed under finite union and arbitrary intersection, $\tau_s^X$ forms a topology on $X^s$.\\ Let $(x_n)_{n\in K}$ be nonthin sequence in $X^s$. Suppose $K_1=\{n\in K: x_n\in X\}$. If $d(K_1)=0$ then the proof is done. Let $\overline{d}(K_1)>0$. If $(x_n)_{n\in K_1}$ has nonthin statistical convergent subsequence in $X$ then the proof will done. Suppose $(x_n)_{n\in K_1}$ has no nonthin statistical convergent subsequence in $X$. From Theorem \ref{th3} it follows that $(x_n)_{n\in K_1}$ statistically converges to $\infty^{X}$.
	\end{proof}
	\begin{note}
		Let's call $(X^s,\tau_s^X)$ one point statistical compactification of $(X,\tau)$.
	\end{note}
	\begin{theorem}
		Let $X$ be statistical sequential. Then $X^s$ is statistical sequential. Also, $\Delta X^s$ is statistically closed if $\Delta X$ is statistically closed in addition.
	\end{theorem}
	\begin{proof}
		Let $Y$ be statistically closed subset of $X^s$. Then $Y\cap X$ is statistically closed in $X$ and so closed in $X$. Therefore $Y$ is closed in $X^s$ if $\infty^{X}\in Y$. Suppose $\infty^{X}\notin Y$. Then $Y$ is closed in $X$. Since $Y$ is statistically closed in $X^s$, $Y$ is statistically compact in $X^s$ and so in $X$. Therefore $Y$ is closed in $X^s$. \\
		Let $\Delta X$ be statistically closed. Suppose $(x_n)_{n\in M}$ is a nonthin sequence in $\Delta X^s$ which statistically converges to $a\in (X^s\times X^s)\backslash \{(\infty^X,\infty^X)\}$ and $y_n=\pi_1(x_n)$, $n\in M$. Then $\overline{d}(N)>0$ where $N=\{n\in M: x_n\in \Delta X\}$ and one of $\pi_1(a)$ and $\pi_2(a)$ must be in $X$, say $\pi_1(a)\in X$. Since nonthin sequence in $X$ can statistically convergent to atmost one point of $X$, $A_1=\{y_n: n\in N\}\cup \{\pi_1(a)\}$ is statistically compact and so statistically closed in $X$ (see Theorem \ref{th4}). Thus $A_1$ is closed in $X$ because $X$ is statistical sequential. Therefore $A_1$ is closed in $X^s$. Since $(y_n)_{n\in N}$ is statistically convergent to $\pi_2(a)$, $\pi_2(a)\in X$. Consequently $a\in \Delta X$ and so $\Delta X^s$ is statistically closed.
	\end{proof}
	\begin{theorem}
		Let $X$ be an open subspace of a topological space $Y$ such that $\Delta Y$ is statistically closed and let $f:Y\to X^s$ be defined by   
		\begin{center}
			$f(x)=\begin{cases}x,& \text{if }x\in X\\ \infty^X, & \text{otherwise. } \end{cases} $
		\end{center}
		Then $f$ is statistically continuous.
	\end{theorem}
	\begin{proof}
		Suppose $(x_n)_{n\in M}$ is nonthin sequence in $Y$ which statistically converges to $y\in Y$. Since $\{n\in M: x_n\notin U\}=\{n\in M: f(x_n)\notin U\}$ for any subset $U$ of $X$, $(f(x_n))_{n\in M}$ statistically converges to $f(y)$ if $y\in X$. Let $y\in Y\backslash X$ and let $C$ be closed statistically compact subset of $X$. Then $\{n\in M: f(x_n)\notin X^s\backslash C\}=\{n\in M: x_n\in C\}$. Since $C$ is statistically compact and $\Delta Y$ is statistically closed, $d(\{n\in M: x_n\in C\})=0$.
	\end{proof}
	\begin{corollary}
		Let $X$ be an open statistical sequential subspace of a statistically compact space $Y$ such that $\Delta Y$ is statistically closed and $Y\backslash X$ has exactly one point.  Then $f$ is statistical homeomorphism where $f:Y\to X^s$ is the unique bijection which is identity on $X$. 
	\end{corollary}
	\begin{note}
		From Note \ref{n} it follows that $X^s$ is not seperable completely metrizable if $X$ is uncountable.
	\end{note}
	\begin{remark}
		Variant of compactness by using other type of statistical convergence(e.g. T-statistical convergence \cite{11}, $\lambda$-statistical convergence \cite{13}, rough statistical convergence \cite{15} etc.) have been studied and secured similar results as in case of statistical compactness.
	\end{remark}

		\vfill\eject
		
	\end{document}